\documentclass{amsart}
\usepackage{amsthm, amssymb, amsmath, graphicx, amsfonts, pstricks, subcaption, enumitem, tipa, hyperref, svg, cancel, mathtools, textgreek}
\usepackage[disable]{todonotes}
\DeclareUnicodeCharacter{2212}{-}

\theoremstyle{plain} \numberwithin{equation}{section}
\newcounter{dummy} 
\numberwithin{dummy}{section}
\newtheorem{theorem}[dummy]{Theorem}
\newtheorem{lemma}[dummy]{Lemma}
\newtheorem{cor}[dummy]{Corollary}
\newtheorem{proposition}[dummy]{Proposition}
\newtheorem{example}[dummy]{Example}
\newtheorem{conjecture}[dummy]{Conjecture}
\newtheorem{fact}[dummy]{Fact}
\newtheorem*{notation*}{Notation}
\newtheorem{definition}[dummy]{Definition}

\newcommand{\Mod}[1]{\ (\mathrm{mod}\ #1)}
\DeclareMathOperator{\Pal}{Pal}

\makeatletter
\newtheorem*{rep@theorem}{\rep@title}
\newcommand{\newreptheorem}[2]{%
	\newenvironment{rep#1}[1]{%
		\def\rep@title{#2 \ref{##1}}%
		\begin{rep@theorem}}%
		{\end{rep@theorem}}}
\makeatother
\newreptheorem{theorem}{Theorem}
\newreptheorem{lemma}{Lemma}
\newreptheorem{cor}{Corollary}

\makeatletter
\providecommand\@dotsep{5}
\renewcommand{\listoftodos}[1][\@todonotes@todolistname]{%
  \@starttoc{tdo}{#1}}
\makeatother

\begin{document}
\author{Ian M. Banfield}
\address{Mathematisches Institut, Universit\"at Bern, Siedlerstrasse 5, Bern, CH-3012}
\email{ian.matthew.banfield@gmail.com}
\title[Christoffel words and the strong Fox conjecture]{Christoffel words and the strong Fox conjecture for two-bridge knots}
\markboth{Ian M. Banfield}{The Strong Fox Conjecture, Two-Bridge Knots and Christoffel words}
\begin{abstract}
	The trapezoidal Fox conjecture states that the coefficient sequence of the Alexander polynomial of an alternating knot is unimodal. We are motivated by a harder question, the strong Fox conjecture, which asks whether the coefficient sequence of the Alexander polynomial of alternating knots is actually log-concave. Our approach is to introduce a polynomial $\Delta(t)$ associated to a Christoffel word and to prove that its coefficient sequence is log-concave. This implies the strong Fox conjecture for two-bridge knots.
\end{abstract}
\maketitle


\section{Introduction}\label{sec:introduction}
In 1962, Fox asked if the absolute values of the coefficients of the Alexander polynomial $\triangle_K(t) = \sum a_i t^i$  of alternating knots $K \subset S^3$ are unimodal , i.e. if
$$|a_0| \leq |a_1| \leq \dots \leq |a_k| \geq |a_{k+1}| \geq \dots \geq |a_n|.$$
This conjecture has become known as Fox's trapezoidal conjecture. \footnote{Unimodal sequences are also known as trapezoidal sequences though we prefer the more commonly used term ``unimodal''.} Work on this conjecture spans over half a century. It has been confirmed for a number of classes of alternating knots, using techniques ranging from combinatorics to Heegard-Floer theory, see \cite{MR2628154,MR550968, MR802722, MR1988285, https://doi.org/10.48550/arxiv.1307.1578, MR4206898}.

A sequence $(a_n)$ of non-negative real numbers is \textbf{log-concave} if $a_{i-1} a_{i+1} \leq a_i^2$ and a polynomial $f(t) = \sum a_i t^i$ is log-concave if its coefficient sequence $(a_i)$ is log-concave. Establishing log-concavity can be challenging. Indeed, the celebrated recent breakthrough work by Huh and his collaborators on the resolution of the Heron-Rota-Welsh conjecture \cite{MR2904577, MR2983081, MR3286530, MR3862944} - establishing that the characteristic polynomial of a matroid is log-concave - followed a half century of failed attempts. Log-concave sequences are ubiquitious in graph theory and combinatorics \cite{MR1110850}. Examples include real-rooted polynomials, matching and chromatic polynomials, and the binomial coefficients, amongst many others.

Log-concave sequences are unimodal. This naturally suggests the following strengthening of Fox's conjecture, which we call the \textbf{strong Fox conjecture}.

\begin{conjecture}
	Let $K \subset S^3$ be an alternating knot. Then the Alexander polynomial $\Delta_K(t) = \sum a_i t^i$ is log-concave.
\end{conjecture}

\begin{figure}[h]
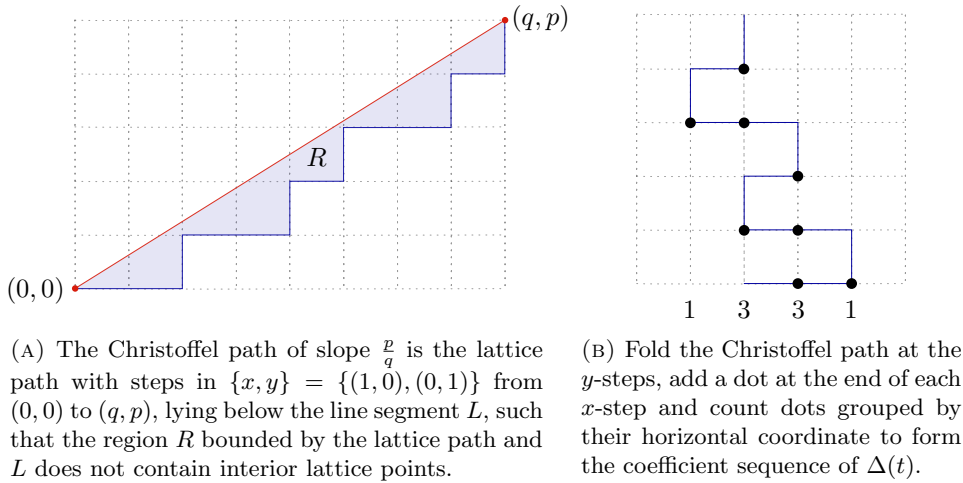

	\setbox0=\hbox{\includesvg{figures/Christoffel_word_and_polynomial_example_06_Folded_Christoffel_path_large.svg}}%
	\subcaptionbox{The Christoffel path of slope $\frac{p}{q}$ is the lattice path with steps in $\{x,y\} = \{(1,0), (0,1)\}$ from $(0,0)$ to $(q,p)$, lying below the line segment $L$, such that the region $R$ bounded by the lattice path and $L$ does not contain interior lattice points. \label{subfig:christoffel_intro_path}}[.56\textwidth]{
		\raisebox{\dimexpr\ht0-\height}{\includesvg{figures/Christoffel_word_and_polynomial_example_05_Christoffel_path,_large.svg}}
	}\hfill%
	\subcaptionbox{Fold the Christoffel path at the $y$-steps, add a dot at the end of each $x$-step and count dots grouped by their horizontal coordinate to form the coefficient sequence of $\Delta(t)$.\label{subfig:christoffel_intro_folded_path_and_delta}}[.40\textwidth]{
		\raisebox{\dimexpr\ht0-\height}{\includesvg{figures/Christoffel_word_and_polynomial_example_06_Folded_Christoffel_path_large.svg}}
	}%
	\caption{The word $w$ encoding the Christoffel path is called the Christoffel word. For $(p,q) = (5,8)$, $w = xxyxxyxyxxyxy$ with polynomial $\Delta_w(t) = t^{-1} + 3 + 3t + t^2$, which is log-concave.}
	\label{figure:christoffel_word_intro}
\end{figure}

Our approach to the strong Fox conjecture is through the theory of Christoffel words \cite{berstelcombinatorics}. Each pair $(p,q)$ of coprime non-negative integers determines a word $w$ on the alphabet $\{x,y\}$, called the Christoffel word of slope $\frac{p}{q}$. The word $w$ encodes a lattice path, see Figure $\ref{subfig:christoffel_intro_path}$. We define a polynomial $\Delta_w(t)$ by decorating this path with dots at the terminal point of every $x$-step, folding it at the $y$-steps, and then read off the coefficient sequence of $\Delta_w(t)$ by counting dots according to their horizontal coordinate as pictured in Figure \ref{subfig:christoffel_intro_folded_path_and_delta}.

Christophe Reutenauer discovered that the polynomial $\Delta_w(t)$ can be recovered through the following construction. Consider the mapping $\psi$ to the multiplicative monoid of $3 \times 3$ matrices (with entries Laurent polynomials over $\mathbb{Z}$), defined by
\begin{equation*}
	x \xmapsto{\psi} \begin{pmatrix}
		t & 0 & t \\
		0 & t^{-1} & t^{-1} \\
		0 & 0 & 1
	\end{pmatrix}
	\mbox{~and~}
	y \xmapsto{\psi}  \begin{pmatrix}
		0 & 1 & 0 \\
		1 & 0 & 0 \\
		0 & 0 & 1
	\end{pmatrix}.
\end{equation*}

\begin{reptheorem}{thm:delta_via_matrices}
	The entry in position $(1,3)$ of the matrix $\psi(w)$ is $\Delta_w(t)$.
\end{reptheorem}

The main result of this article is Theorem \ref{thm:maintheorem}. To prove it, we use recent work by Gross-Mansour-Tucker-Wang \cite{MR3355766} on combinations of log-concave sequences.

\begin{reptheorem}{thm:maintheorem}
	Let $w$ be a Christoffel word. The polynomial $\Delta_w(t)$ is log-concave.
\end{reptheorem}

This provides a positive answer to the strong Fox conjecture for a well-known and frequently studied class of alternating knots, namely knots with bridge number two. A bridge in a knot diagram is an arc with at least one overcrossing and the bridge number of a knot is defined to be the minimal number of bridges in any diagram. A diagram of a two-bridge knot, exhibiting the two bridges, is given in Figure \ref{fig:two_bridge_knot_p=5,q=8}.

\begin{figure}[h]
	\centering
	\includesvg[scale=0.75]{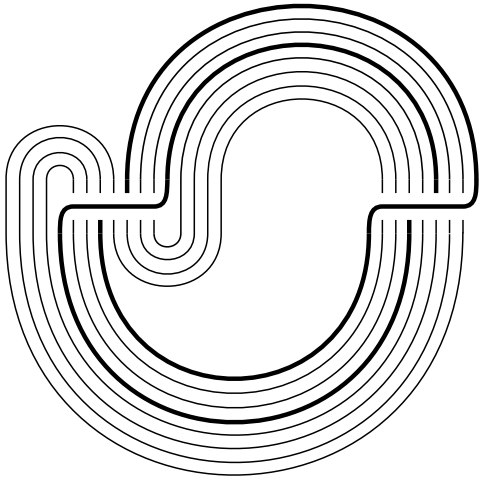}
	\caption{A two-bridge knot, corresponding to $(5,8)$ in Schubert's normal form. The two bridges are indicated in heavy lines.}
	\label{fig:two_bridge_knot_p=5,q=8}
\end{figure}

Inspired by Hoste's interpretation \cite{MR4143690} of a classic result by Minkus \cite{MR643587}, describing how the Alexander polynomial can be computed via an integer walk, we prove an equivalent result for the polynomial $\Delta_w(t)$. This establishes a relation between our polynomial $\Delta_w(t)$ and the Alexander polynomials of two-bridge knots.

\begin{reptheorem}{thm:two_bridge_alex_is_delta_of_christoffel}
	Let $K$ be a two-bridge knot. There exists a Christoffel word $w$ satisfying $\Delta_K(t) = \varsigma(\Delta_w(t))$, where $\varsigma$ is the sign-alternation operator on Laurent polynomials, defined by $t^n \xmapsto{\varsigma} (-1)^n t^n$ and extended linearly.
\end{reptheorem}

Our proof is constructive. Schubert's normal form gives a classification of two-bridge knots by pairs of coprime integers $(p,q)$ satisfying $p$ odd and $0 < p < q$ \cite{MR1417494} and the word $w$ satisfying Theorem \ref{thm:two_bridge_alex_is_delta_of_christoffel} is none other than the Christoffel word of slope $\frac{p}{q}$. The two-bridge knot $K$ in Figure \ref{fig:two_bridge_knot_p=5,q=8} and the Christoffel word from Figure \ref{subfig:christoffel_intro_path} both correspond to $(p,q) = (5,8)$ in their respective constructions, and therefore $\Delta_K(t) = \varsigma(\Delta_w(t)) = -t^{-1} + 3 - 3t + t^2$. 

\begin{repcor}{cor:strong_fox_for_two_bridge}
	The strong Fox conjecture is true for two-bridge knots.
\end{repcor}

\subsection*{Outline of this article}
In Section \ref{sec:christoffel}, we define Christoffel words and discuss a characterization of Christoffel words as palindromic closures. In Section \ref{sec:log_concave} we recall two relations on log-concave sequences, the synchronicity and ratio-dominance relations of \cite{MR3355766}. The key property of these relations is that sums of related sequences are log-concave. The definition of $\Delta_w(t)$ and the proof of Theorem \ref{thm:maintheorem} is explained in Section \ref{sec:alexander_polynomial_log_concave}. We end by explaining how Theorem \ref{thm:maintheorem} implies the log-concavity of the Alexander polynomial of two-bridge knots in Section \ref{sec:strongfox}.


\section{Christoffel words}
\label{sec:christoffel}
\subsection{Definitions and Notations}
We follow the conventions from \cite{berstelcombinatorics}. For convenience, we start by briefly recalling the key definitions used in this article.

An \textbf{alphabet} $A$ is a finite set of symbols. The elements of $A$ are called \textbf{letters}. A \textbf{word} over the alphabet $A$ is an element in the free monoid $A^*$ generated by the alphabet $A$. The identity element of $A^*$ is called the \textbf{empty word} and denoted $\epsilon$.

Let $w \in A^*$ be a word. Then there exist unique $a_1, \dots a_r \in A$ and a unique $r \geq 0$ such that $w = a_1 a_2 \dots a_r$. The integer $r$ is called the \textbf{length} of the word $w$. A \textbf{factor} of $w$ is a word $v = a_i a_{i+1} \cdots a_j$ for $1 \leq i \leq j \leq r$ and a \textbf{prefix} of $w$ is a factor of the form $v = a_1 \cdots a_j$ for $1 \leq j$. The \textbf{reversal} of $w$ is the word $\widetilde{w} = a_r a_{r-1} \dots a_1$. A \textbf{palindrome} is a word $w \in A^*$ that equals its reversal.

Let $A,B$ be two alphabets. A \textbf{morphism} is a map $f : A^* \to B^*$ satisfying $f(uv) = f(u)f(v)$ for all $u,v \in A^*$, i.e. a morphism in the category of free monoids.

\subsection{Christoffel words}

\begin{definition}
	Let $p,q$ be non-negative, coprime integers. The Christoffel path of slope $\frac{p}{q}$ is the lattice path in the integer lattice $\mathbb{Z}^2$ from $(0,0)$ to $(p,q)$ with steps $x = (1,0)$ and $y = (0,1)$ satisfying the following two conditions.
	\begin{enumerate}
		\item The lattice path lies below the line segment $L$ with endpoints at $(0,0)$ and $(p,q)$ and,
		\item the region bordered by the line segment $L$ and the lattice path does not contain interior lattice points.
	\end{enumerate}
	The Christoffel path is encoded by a word $w \in \{x,y\}^*$, the so-called Christoffel word.
\end{definition}

A Christoffel word is trivial if $w = x$ or $w = y$. In the rest of the article, by Christoffel word we mean a non-trivial Christoffel word. Christoffel words admit a number of interesting characterizations. We will only discuss the palindromic characterization and refer to \cite{berstelcombinatorics} for a more detailed discussion of Christoffel words and their various characterizations and further references. In order to state the palindromic characterization we first explain the so-called iterated palindromic closure.

\begin{definition}[\cite{MR1468450}]
	\label{def:palindromic_closure}
	Suppose $w = uv \in \{x,y\}^*$, where $v$ is the longest suffix of $w$ that is a palindrome. Let $w^{+} = w\widetilde{u}$, the so-called right palindromic closure of $w$. The iterated palindromic closure is the operation $\Pal : \{x,y\}^* \to \{x,y\}^*$, defined recursively by $\Pal(\epsilon) = \epsilon$ and $\Pal(zv) = (\Pal(v)z)^{+}$ for $z \in \{x,y\}$. 
\end{definition}

The pictorial interpretation of the iterated palindromic closure is as follows. The letter $z \in \{x,y\}$ acts on the image of $\Pal$ by increasing the length of the horizontal (if $z = x$) or vertical (if $z = y$) segments of the path corresponding to the word $u \in \Pal(\{x,y\}^*)$ by one, see Figure \ref{fig:action_of_x}. This extends to a monoid action of $\{x,y\}^*$ on the image $\Pal(\{x,y\}^*)$ by defining
\begin{equation}
	w \cdot \Pal(v) \coloneqq \Pal(wv).
\end{equation}

An immediate implication of this interpretation is the following algorithm to calculate the inverse. Let $u \in \Pal(\{x,y\}^*)$. The path corresponding to $u$ starts with either $n$ horizontal (vertical) segments; shorten all horizontal (vertical) segments by $n$ and note down $x^n$ ($y^n$). Continue until the path is trivial. The word noted down in this process is the inverse of $u$.

\begin{figure}
	\centering
	\includesvg{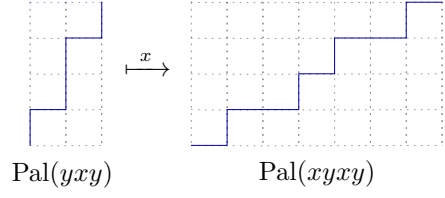}
	\caption{The action of $x$ on $\Pal(\{x,y\}^*)$: The lengths of all horizontal segments, including those of length 0, are increased by 1. The action of $y$ similarly lengthens vertical segments.}
	\label{fig:action_of_x}
\end{figure}

\begin{fact}[\cite{berstelcombinatorics}, Proposition 4.9, Proposition 4.14]
	\label{fact:palindromic_char_of_Christoffel}
	 A word $w \in \{x,y\}^*$ is a Christoffel word if and only if $w = x\Pal(v)y$ for some $v \in \{x,y\}^*$.
\end{fact}
We refer to $u = \Pal(v)$ in a Christoffel word $w = xuy$ as the \textbf{palindromic factor} of $w$. Fact \ref{fact:palindromic_char_of_Christoffel} allows for ``extending'' the monoid action of $\{x,y\}^*$ from iterated palindromic closures to Christoffel words via
\begin{equation}
	\label{eq:monoid_action_on_christoffel}
	w \cdot x\Pal(v)y \coloneqq x(w \cdot \Pal(v))y
\end{equation}
\begin{example}
	The Christoffel word of slope $\frac{5}{8}$ is
	\begin{align*}
		w &= x(xyxxyxyxxyx)y = x\Pal(xyxy)y. \\
		\intertext{Acting by $yx$ on $w$ gives}
		(yx) \cdot w &= x\left[ (yx) \cdot \Pal(xyxy) \right]y \\
						&= x\left[ y \cdot (xxyxxxyxxyxxxyxx) \right]y \\
						&= x(yxyxyyxyxyxyyxyxyyxyxyxyyxyxy)y.
	\end{align*}
\end{example}

\section{The synchronicity and ratio-dominance relations for log-concave sequences}
\label{sec:log_concave}
In \cite{MR3355766}, Gross-Mansour-Tucker-Wang introduced two relations on log-concave sequences to prove log-concavity of the genus distributions of some graphs. We will briefly recall these relations and various facts that we use to prove Theorem \ref{thm:maintheorem}.
\begin{definition}
	\label{def:synced_sequences}
	Two non-negative, log-concave sequences $A = (a_n)$ and $B = (b_n)$ are synchronized, denoted $A \sim B$, if $a_{i-1} b_{i+1} \leq a_i b_i$ and $a_{i+1} b_{i-1} \leq a_i b_i$ for all $i$. 
\end{definition}
The synchronicity relation is reflexive and symmetric, but not transitive, so in particular it is not an equivalence relation. The key property of the synchronicity relation is its behavior under summation, as the following two facts show.
\begin{fact}[\cite{MR3355766}, Theorem 2.2]
	\label{fact:sum_of_synced_is_log_concave}
	Let $A$ and $B$ be synchronized sequences and let $u,v \geq 0$. The sequence $uA + vB$ is log-concave.
\end{fact}
\begin{fact}[\cite{MR3355766}, Lemma 2.4]
	\label{fact:transitive-ish_synced_sequence}
	Let $A \sim B$, $B \sim C$, $A \sim C$. Then $A + B \sim C$.
\end{fact}
\begin{example}
	Let $F_7 = (1,5,6,1)$ and let $F_6 = (0,1,4,3)$. A straightforward verification shows that these sequences are log-concave and synchronized. The sequence $F_7 + F_6 = (1,6,10,4)$ is log-concave by Fact \ref{fact:sum_of_synced_is_log_concave}.
\end{example}
\begin{example}
	\label{ex:fibonnaci_polynomial}
	The sequences in the previous example are the coefficient sequences of Fibonacci polynomials, defined by the following equations.
	$$F_0 = 0, F_1 = 1, F_{i+1} = xF_i + F_{i-1}.$$
	The Fibonnaci polynomial $F_i$ is log-concave (by this we mean that the coefficients of the even or odd powers form a log-concave sequence). This is classical; an elegant proof via the synchronicity relation and Fact \ref{fact:transitive-ish_synced_sequence} can be given as follows. Note that $F_i \sim x^2 F_i$ if $F_i$ is log-concave. An easy inductive argument then establishes that $xF_i \sim F_{i-1}$ and $F_i \sim x F_{i-1}$ for all $i$. In particular, $F_i$ is log-concave.
\end{example}
Let $A = (a_i)$ and $B = (b_i)$ be synchronized sequences and consider the ratios $\alpha_i = \frac{a_i}{a_{i-1}}$ and  $\beta_i = \frac{b_i}{b_{i-1}}$. The log-concavity of $A$ and $B$ can be restated as $(\alpha_i)$ and $(\beta_i)$ are increasing sequences. The condition that $A$ and $B$ are synchronized is equivalent to $\beta_{i+1} \leq \alpha_i$ and $\alpha_{i+1} \leq \beta_i$. Thus $A$ and $B$ are synchronized sequences exactly if $\{\alpha_{i+1}, \beta_{i+1}\} \leq \{\alpha_i, \beta_i\}$, where $S \leq T$ means that the inequality holds for all pairs of elements in these sets.
\begin{definition}
	Let $A = (a_i)$ and $B = (b_i)$ be synchronized sequences. The sequence $B$ is ratio-dominant over $A$, denoted $A \lesssim B$, if $a_{i+1} b_i \leq a_i b_{i+1}$ for all $i$.
\end{definition}
Thus $B$ is ratio-dominant over $A$ precisely if $\alpha_i \leq \beta_i$. The ratio-dominance relation is reflexive and transitive as long as the sequences are synchronized.
\begin{notation*}
	In Sections \ref{sec:alexander_polynomial_log_concave} and \ref{sec:strongfox}, we will frequently discuss the synchronicity and ratio-dominance of the coefficient sequences of polynomials, say $f(t)$ and $g(t)$. For brevity we omit the ``coefficient sequence of'' and simply state this as $f \sim g$ and $f \lesssim g$. If $f$ or $g$ are Laurent polynomials and the lowest degree of either polynomial is $t^{-k}$, we define the coefficient sequences of $f$ and $g$ to be the ordinary coefficient sequences of the polynomials $t^k f$ and $t^k g$.
\end{notation*}

\subsection{Convolution of sequences}
\label{subsec:convolution}
In this section we recall the convolution of two sequences and explain how the synchronicity and ratio-dominance relation behave under convolution.
\begin{definition}
	Let $A = (a_i)$ and $B = (b_i)$ be sequences. The convolution of $A$ and $B$ is the sequence $A*B = (c_i)$ with coefficients
	$$c_i = \sum_{j = 0}^i a_j b_{i-j}.$$
\end{definition}
That is, if $A$ ($B$, respectively) is the coefficient sequence of the polynomial $f(t)$ ($g(t)$, respectively) then $A*B$ is the coefficient sequence of the product $(fg)(t)$.
\begin{fact}[\cite{MR3355766}, Theorem 2.8]
	\label{fact:synchronicity_under_convolution}
	Let $A$ and $B$ be synchronized sequences, and let $C$ be log-concave. Then $A * C$ and $B * C$ are synchronized.
\end{fact}
In the next fact, one may encounter expressions of the form $\frac{a}{b} \leq \frac{c}{d}$ where $b = 0$ or $d = 0$. In these cases, we define the inequality to be true if $a=b=0$, or $c=d=0$ or $b = d = 0$.
\begin{fact}[\cite{MR3355766}, Theorem 2.20]
	\label{fact:convolution_ratio_dominance_and_lexiocographic}
	Let $\mathcal{A}_i = (a_{i,t})$, $\mathcal{B}_i = (b_{i,t})$, $\mathcal{W}_i = (w_{i,t})$ be finite sequences of sequences (so e.g. $a_{i,t}$ is the $t$-th term of the sequence $\mathcal{A}_i$). Suppose that $W_i \lesssim W_j$ whenever $i < j$ and suppose further that $\mathcal{A}$ and $\mathcal{B}$ satisfy the following two conditions.
	\begin{align}
		\label{eq:lexiocographic_function_1}
		 \frac{b_{1,t}}{a_{1,t}} \leq \frac{b_{2,t}}{a_{2,t}} \leq \dots \leq \frac{b_{n,t}}{a_{n,t}} \leq \frac{b_{1,t+1}}{a_{1,t+1}}, \mbox{ and } \\
		 \label{eq:lexiocographic_function_2}
		 \frac{a_{1,t-1}}{b_{1,t}} \leq \frac{a_{2,t-1}}{b_{2,t}} \leq \dots \leq \frac{a_{n,t-1}}{b_{n,t}} \leq \frac{a_{1,t}}{b_{1,t+1}}
	\end{align}
	Then $$\sum_{i=1}^n \mathcal{W}_i * \mathcal{A}_i \lesssim \sum_{i=1}^n \mathcal{W}_i * \mathcal{B}_i.$$
\end{fact}

\section{The polynomial for Christoffel words}
\label{sec:alexander_polynomial_log_concave}
This section introduces the polynomial $\Delta(t)$. We first give the definition of $\Delta(t)$ and then proceed to discuss the structural properties, which will enable us to prove that $\Delta(t)$ is log-concave.

\subsection{Folded (Christoffel) paths}
\label{subsec:def_of_delta}
A Christoffel word $w$ determines a lattice path in $\mathbb{Z}^2$, with steps in $\{x,-x,y\}$, where $x = (1,0), -x = (-1,0), y = (0,1)$, as pictured in Figure \ref{subfig:christoffel_intro_folded_path_and_delta}. We refer to this path as the folded Christoffel path. The polynomial $\Delta(t)$ will be defined more generally for arbitrary lattice paths with steps in $\{x,-x,y\}$. We start by formalizing the operation of ``folding'' a word over the alphabet $\{x,y\}$, e.g. a Christoffel word, to a path with steps in $\{x,-x,y\}$.

\begin{definition}
	\label{def:folded_path}
	Let $p$ be a lattice path in $\mathbb{Z}^2$. We say that a step of $p$ starts at odd height if the initial point of the step is at $(a,b) \in \mathbb{Z}^2$ with $b$ odd. Let $v \in \{x,y\}^*$ and let $p_v$ be the corresponding lattice path in $\mathbb{Z}^2$ on the step set $\{x,y\}$, starting at the origin. The folded path $F(v)$ is the lattice path in $\mathbb{Z}^2$ on the step set $\{x,-x,y\}$ obtained from $p_v$ by replacing each $x$-step that starts at odd height by the step $-x$.
\end{definition}
\begin{example}
	The folded path $F(v)$ for $v = xxxyyyxyyxxyxxxxyxxyyy$ is pictured in Figure \ref{fig:definition_polynomial_g}.
\end{example}
The path $p_v$ corresponding to a word $v \in \{x,y\}^*$ is injective and it is easy to see that the folded path $F(v)$ is also injective. This allows us to refer to \textit{the} point $(a,b)$ on a path $p_v$ or $F(v)$. The folding operator itself is clearly injective.

\begin{definition}
	\label{def:folded_christoffel} Let $w$ be the Christoffel word of slope $\frac{p}{q}$. The folded Christoffel path is the path $F(w)$ and the folded Christoffel word is the word over the alphabet $\{x,-x,y\}$ encoding the folded Christoffel path.
\end{definition}

\begin{definition}
	\label{def:alex_for_a_path}
	Let $p$ be a lattice path in $\mathbb{Z}^2$ with steps in $\{x,-x,y\}$ and suppose the terminal points of the $\pm x$-steps of the path are $T_p = \{ (i,j) \}$. The polynomial $\Delta_p(t) \in \mathbb{Z}[t,t^{-1}]$ is defined by
	$$\Delta_p(t) = \sum_{k=-\infty}^{\infty} c_k t^k,$$
	where $c_k$ is the count of points $(k,*) \in T_p$.
\end{definition}
Note that the initial point of the path $p$ is an integral part of the definition of $\Delta_p(t)$ and not assumed to be the origin. The value of $\Delta_p(1)$ is the number of $\pm x$ steps of the path $p$; for folded paths $\Delta_{F(v)}(1)$ recovers the $x$-length of the word $v$. We indicate the terminal points of the $\pm x$-steps of the lattice path $p$ by decorating the path with solid black dots as in Figure \ref{figure:christoffel_word_intro}. Let $w$ be a Christoffel word. We write $\Delta_w(t) = \Delta_{F(w)}(t)$ for the polynomial and $T_w = T_{F(w)}$ for the terminal points in this case. This will cause no confusion as we use the symbol $w$ exclusively to refer to Christoffel words.

The polynomial $\Delta$ counts the positions of the $\pm x$-steps of the lattice path. This motivates the following definition of a polynomial which counts the positions of the $yy$-steps of the lattice path, up to a shift.

\begin{definition}
	Let $p$ be a lattice path in $\mathbb{Z}^2$ with steps in $\{x,-x,y\}$. Consider the set of $y$-steps of the path that are immediately followed by another $y$-step and let $SC_p = \{ s(i,j) \}$ be the ``shifted'' terminal points of such steps, where $s(i,j) = (i-1,j)$ if $j$ is odd and $s(i,j) = (i,j)$ if $j$ is even. The polynomial $G_p(t) \in \mathbb{Z}[t,t^{-1}]$ is defined by
	$$G_p(t) = \sum_{k=-\infty}^{\infty} c_k t^k,$$
	where $c_k$ is the count of points $(k,*) \in SC_p$.
\end{definition}

The (shifted) terminal point of a $y$-step immediately followed by another $y$-step is the (shifted) center point of the matching $yy$-step. For this reason, we refer to $SC_p$ as the \textbf{shifted center points} of the $yy$-steps of $p$. We indicate these points with hollow green dots on the lattice path as in Figure \ref{fig:definition_polynomial_g}. If $w$ is a Christoffel word, then $G_w(t) = G_{F(w)}(t)$ denotes the polynomial for the folded Christoffel path for $w$. If the Christoffel word is of slope less than one then $G_w(t) = 0$, as the Christoffel path and consequently the folded Christoffel path do not contain any $yy$-steps. Indeed, for Christoffel words $G_w(t) \neq 0$ if and only if $w = x\Pal(yv)x$ for a word $v \in \{x,y\}^*$.

\begin{figure}
	\captionsetup{margin=0in}
	\captionbox{$G_{F(v)}(t) = t + 3t^2 + t^3$. \label{fig:definition_polynomial_g}}[.5\textwidth]{
		\includesvg{figures/Proof_of_main_theorem_01_Definition_of_polynomial_G.svg}
	}%
	\captionbox{The $yy$-reduction.\label{fig:yy-reduction}}[.5\textwidth]{
		\includesvg{figures/Proof_of_main_theorem_13_yy_reduction.svg}
	}
	\label{fig:definition_of_folded_path,g,yy-reduction}
\end{figure}

Let $p$ be a lattice path with steps in $\{x, -x, y\}$. We refer to the terminal points of the $\pm x$-steps and the shifted center points of the $yy$-steps as \textbf{dots contributing} to $\Delta_p(t)$ and $G_p(t)$, respectively. Note that there is a natural order on the dots, namely the order in which they are encountered as one traverses the path $p$. We identify a dot $d=(a,b)$ with the corresponding term $t^a$ of the polynomial $\Delta_p(t)$ or $G_p(t)$. The degree of the identified monomial is the horizontal coordinate of $d$ and accordingly the \textbf{degree of a dot} is defined as $\deg(d) = a$. The implied expressions for the polynomials $\Delta_p(t)$ and $G_p(t)$ as sums over dots are 
\begin{equation}
	\label{eq:polynomial_as_dots}
	\Delta_p(t)= \sum_{d \in T_p} t^{\deg(d)} \mbox{, and ~} G_p(t)= \sum_{d \in SC_p} t^{\deg(d)}.
\end{equation}

\subsection{Properties of folded paths}
\label{subsec:basic_properties_of_delta}
We now investigate the properties of folded paths and the polynomials $\Delta(t)$ and $G(t)$. The first result, Lemma \ref{lem:delta_and_g_of_a_concatenation}, is a foundational technical lemma describing the behavior of the polynomials under the concatenation of words. We then define and study an operation called $yy$-reduction, defined for Christoffel words of slope $1 < \frac{p}{q} < 2$, and calculate the polynomial for a $yy$-reduced word.

\begin{lemma}
	\label{lem:delta_and_g_of_a_concatenation}
	Let $u_1,u_2 \in \{x,y\}^*$ and let $(a,b)$ be the terminal point of the folded path $F(u_1)$. Then
	\begin{equation}
		\label{eq:delta_under_word_concatenation}
		\Delta_{F(u_1 u_2)}(t) = \Delta_{F(u_1)}(t) + t^a \begin{cases}
			\Delta_{F(u_2)}(t) & \text{for~} b \text{~even}\\
			\Delta_{F(u_2)}(t^{-1}) & \text{for~} b \text{~odd}
		\end{cases}.
	\end{equation}
	If either the last letter of $u_1$ or the first letter of $u_2$ is $x$, then
	\begin{equation}
		\label{eq:g_under_word_concatenation}
		G_{F(u_1 u_2)}(t) = G_{F(u_1)}(t) + t^a \begin{cases}
			G_{F(u_2)}(t) & \text{for~} b \text{~even}\\
			t^{-1} G_{F(u_2)}(t^{-1}) & \text{for~} b \text{~odd}
		\end{cases}.
	\end{equation}
\end{lemma}
\begin{proof}
	The word $s \in \{x,-x,y\}^*$ encoding a folded path $F(v)$ is given by replacing every occurence of $x$ in $v$ that is preceeded by an odd number of occurences of $y$ by $-x$. So let $s$, $s_1, s_2$ be the words encoding $F(u_1 u_2)$, $F(u_1)$ and $F(u_2)$, respectively and let $k,l$ be the number of occurences of $x$ in $u_1, u_2$, respectively. Clearly $s = s_1 s'$. The number of occurences of $y$ in $u_1$ equals $b$. This implies that for $1 \leq i \leq l$, the number of occurences of $y$ preceeding the $(k+i)$-th occurence of $x$ in $u_1 u_2$ equals $b$ + the number of occurences of $y$ preceeding the $i$-th occurence of $x$ in $u_2$. Thus there are two cases: Either $b$ is even and $s' = s_2$ or $b$ is odd and $s'$ is given by replacing all factors of $\pm x$ in $s_2$ by $\mp x$. In the first case, $F(u_1 u_2) = F(u_1) * F(u_2)$ and in the second case, $F(u_1 u_2) = F(u_1) * \tau_* F(u_2)$, where $\tau$ is the reflection along $x=0$ as this reflection flips $\pm x$-steps to $\mp x$-steps, preserves $y$-steps and fixes the origin.

	Equation \ref{eq:polynomial_as_dots} and elementary plane geometry imply that (1) if a lattice path $r'$ is the reflection of $r$ along the vertical line $x=\alpha$, then $\Delta_{r'}(t) = t^{2\alpha} \Delta_r(t^{-1})$, (2) if a lattice path $r'$ is a translation of $r$ by $(\alpha, \beta)$ then $\Delta_{r'}(t) = t^\alpha \Delta_r(t)$ and (3) if a lattice path is a concatenation $r = r_1 * r_2$ with the terminal point of $r_1$ being $(\alpha,\beta)$ then $\Delta_r(t) = \Delta_{r_1}(t) + t^\alpha \Delta_{r_2}(t)$. Equation \ref{eq:delta_under_word_concatenation} thus follows from the expression of $F(u_1 u_2)$ as a concatenation of paths involving $F(u_1)$ and $F(u_2)$.

	If either the last letter of $u_1$ or the first letter of $u_2$ is $x$, then all monomials in $G_{F(u_1 u_2)}(t)$ correspond to occurences of $yy$ in either $u_1$ or $u_2$. Formulae similar to those stated in the previous paragraph for $\Delta_r(t)$ hold for $G_r(t)$. Indeed, to prove Equation \ref{eq:g_under_word_concatenation} it is sufficient to consider translations by even vertical distances and reflections composed with translations by odd vertical distances and to observe that the shift map $s(i,j)$ in the definition of the polynomial $G_p(t)$ is well-behaved under such transformations. Equation \ref{eq:g_under_word_concatenation} follows from the expression of $F(u_1 u_2)$ as a concatenation.
\end{proof}
\todo{Draw an example of a folded path as a concatenation of two paths and show Equation \ref{eq:delta_under_word_concatenation} and Equation \ref{eq:g_under_word_concatenation} holds.}
\begin{cor}
	\label{cor:delta_invariant_under_killing_yy}
	Let $u_1, u_2 \in \{x,y\}^*$. The folded paths $F(u_1 yy u_2)$ and $F(u_1 u_2)$ satisfy
	$$\Delta_{F(u_1 yy u_2)}(t) = \Delta_{F(u_1 u_2)}(t).$$
\end{cor}
\begin{proof}
	Let $(a,b)$ be the terminal point of $F(u_1)$. Then $(a,b+2)$ is the terminal point of $F(u_1yy)$. Further, $\Delta_{F(u_1)}(t) = \Delta_{F(u_1yy)}(t)$, by Lemma \ref{lem:delta_and_g_of_a_concatenation} as $\Delta_{F(yy)}(t) = 0$. The claim now follows from applying Lemma \ref{lem:delta_and_g_of_a_concatenation} once more.
\end{proof}

\begin{cor}
	\label{cor:delta_of_xyu_vs_xu}
	Let $u \in \{x,v\}^*$. Then $\Delta_{F(xyu)}(t) = t^2 \Delta_{F(xu)}(t^{-1})$.
\end{cor}
\begin{proof}
	Note that $\Delta_{F(x)}(t) = \Delta_{F(xy)}(t) = t$. Using Lemma \ref{lem:delta_and_g_of_a_concatenation} one computes $\Delta_{F{xu}}(t) = t + t \Delta_{F(u)}(t)$ and another application of Lemma \ref{lem:delta_and_g_of_a_concatenation} yields
	\begin{align*}
		\Delta_{F(xyu)} &= t + t \Delta_{F(u)}(t^{-1}) \\
						&= t^2 (t^{-1} + t^{-1} \Delta_{F(u)}(t^{-1})) \\
						&= t^2 \Delta_{F(xu)}(t^{-1}). \qedhere
	\end{align*}
\end{proof}

We now give another description of the polynomial $\Delta_{F(v)}(t)$. This result was explained to us by Christophe Reutenauer, who obtained it using the theory of rational non-commutative series \cite{berstel2011noncommutative}.

\begin{theorem}
	\label{thm:delta_via_matrices}
	Let $M_3(\mathbb{Z}[t,t^{-1}])$ be the multiplicative monoid of $3 \times 3$-matrices whose entries are Laurent polynomials over the integers. Define the morphism of monoids $\psi : \{x,y\}^* \to M_3(\mathbb{Z}[t,t^{-1}])$ by
	\begin{equation*}
		x \mapsto \begin{pmatrix}
			t & 0 & t \\
			0 & t^{-1} & t^{-1} \\
			0 & 0 & 1
		\end{pmatrix}
		\mbox{~and~}
		y \mapsto \begin{pmatrix}
			0 & 1 & 0 \\
			1 & 0 & 0 \\
			0 & 0 & 1
		\end{pmatrix}.
	\end{equation*}
	Then the polynomial $\Delta_{F(v)}(t)$ is given as the $(1,3)$ entry of the matrix $\psi(v)$.
\end{theorem}
\begin{proof}
	First note that the morphism $\psi$ is well-defined as $\{x,y\}^*$ is free. The result is immediate for words $v = y^n$, so suppose otherwise. Let $d_1, \dots, d_k$ be the dots contributing to $\Delta_{F(v)}(t)$, let $d = \deg(d_k)$ and let $|v|_y$ denote the number of occurences of $y$ in $v$. We claim that the matrix $\psi(v)$ is of form
	$$\psi(v) =  \begin{pmatrix}
		t^d & 0 & \Delta_{F(v)} \\
		0 & t^{-d} & \star \\
		0 & 0 & 1
	\end{pmatrix} \mbox{~or~} \psi(v) = \begin{pmatrix}
		0 & t^d & \Delta_{F(v)} \\
		t^{-d} & 0 & \star \\
		0 & 0 & 1
	\end{pmatrix},$$
	where the first form occurs if $|v|_y$ is even and the second if $|v|_y$ is odd. We present an elementary proof by induction on $n = |v|$. This is immediate if $n=1$. Suppose that the claim holds for all words $v'$ of length at most $n$ and let $v$ be of length $n+1$. If $v = v'y$ then Lemma \ref{lem:delta_and_g_of_a_concatenation} implies that $\Delta_{F(v')} = \Delta_{F(v)}$. Therefore, the claim follows from the fact that $\psi(y)$ is the permutation matrix that, acting on the right, exchanges the two forms. If $v = v'x$ then there is one additional dot $d_{k+1}$ contributing to $\Delta_{F(v)}$. Denote the degree of this dot by $d' = \deg(d_{k+1})$ and note that $d' = d \pm 1$, depending on whether $d_{k+1}$ is the terminal point of a $\pm x$-step, i.e. depending on whether $|v|_y$ is even ($+1$) or odd ($-1$). If $|v|_y$ is even, 
	\begin{align*}
		\psi(v) = \psi(v')\psi(x)
	= \begin{pmatrix}
		t^{d+1} & 0 & t^{d+1} + \Delta_{F(v')} \\
		0 & t^{-d-1} & \star \\
		0 & 0 & 1
	\end{pmatrix}
	= \begin{pmatrix}
		t^{d'} & 0 & \Delta_{F(v)} \\
		0 & t^{d'} & \star \\
		0 & 0 & 1
	\end{pmatrix},
	\end{align*}
	The case $|v|_y$ odd is similar. This completes the induction.
\end{proof}

\begin{definition}
	Let $w = x\Pal(yxv)y$ for $v \in \{x,y\}^*$, i.e. $w$ is Christoffel word of slope $1 < \frac{p}{q} < 2$ and let $u = \Pal(yxv)$ be the palindromic factor of $w$. Define the \textbf{$yy$-reduction of $w$} to be the word $r(w) = xu'y$, where $u'$ is obtained from $u$ by deleting the leading and trailing $y$, and deleting all occurences of $yy$.
\end{definition}
\begin{example}
	This example is pictured in Figure \ref{fig:yy-reduction}. Let $w = x\Pal(y xx yy)y$. The palindromic factor of $w$ is
	\begin{align*}
		u &= \Pal(y xx yy) = y x y x yy x y x yy x y x y. \\
		\intertext{Striking the leading and trailing $y$, and all occurences of $yy$ yields}
		u' &= \cancel{y}x y x \cancel{yy} x y x \cancel{yy} x y x \cancel{y} = xyxxyxxyx,
	\end{align*}
	and therefore $r(w) = x (xyxxyxxyx) y = x \Pal(xyxx) y$.
\end{example}

We next show in Lemma \ref{lemma:word_of_yy_reduction} that a $yy$-reduced word is a Christoffel word. Indeed, let the slope of $w$ be $1 < \frac{p}{q} < 2$. By counting the occurences of $x$ and $y$ in $r(w)$ one notes that the slope of the $yy$-reduced word $r(w)$ is $0 < -\frac{p}{q} + 2 < 1$ (though a priori, $r(w)$ may not be a Christoffel word). A corollary of the proof of Lemma \ref{lemma:word_of_yy_reduction} is that in the expressions $w = x\Pal(v)y$ and $r(w) = x\Pal(v')y$, the length of $v$ is strictly less than the length of $v'$. This will form part of the inductive argument in the proof of Theorem \ref{thm:maintheorem}.

\begin{lemma}
	\label{lemma:word_of_yy_reduction}
	Let $w = x\Pal(yxv)y$ for $v \in \{x,y\}^*$ and let $E : \{x,y\}^* \to \{x,y\}^*$ be the morphism exchanging $x$ and $y$. The $yy$-reduced word is $r(w) = x\Pal(xE(v))y$.
\end{lemma}
\begin{proof}
	Let $u = \Pal(yxv)$ be the palindromic factor of $w$. Then $u$ can be written as $$u = y (xyy)^{r_1} xy \cdots (xyy)^{r_n} xy,$$ where $r_i \geq 0$. Deleting the leading and trailing $y$ and deleting all occurences of $yy$ gives
	$$u' = x^{r_1 + 1}y \cdots  y x^{r_n + 1} = x \cdot (x^{r_1} y \cdots y x^{r_n}) = x \cdot E(y^{r_1} x \cdots x y^{r_n}). $$
	On the other hand,
	$$
		u = y (xyy)^{r_1} xy \cdots (xyy)^{r_n} xy
		= y \cdot \left ( (xy)^{r_1} x \cdots (xy)^{r_n} \right )
		= yx \cdot \left ( y^{r_1}x \cdots x y^{r_n} ) \right ),
	$$
	and so $r(w) = xu'y = x \Pal(x E(v)) y$.
\end{proof}

\begin{lemma}
	\label{lem:delta_of_yy_reduction}
	Let $w = x\Pal(yxv)y$ for $v \in \{x,y\}^*$. Then $\Delta_w(t) = t^2 \Delta_{r(w)}(t^{-1})$.
\end{lemma}
\begin{proof}
	Write $w = xyuyy$ so that $r(w) = xu'y$, where $u'$ is obtained from $u$ by deleting all occurences of $yy$. The claim follows from repeated applications of Corollary \ref{cor:delta_invariant_under_killing_yy}, and Corollary \ref{cor:delta_of_xyu_vs_xu}.
\end{proof}

The coefficient sequence of a polynomial $f(t^{-1})$ is obtained by reversing the coefficient sequence of $f(t)$. Lemma \ref{lem:delta_of_yy_reduction} thus implies that the coefficient sequence of $\Delta_w(t)$ is the reverse of $\Delta_{r(w)}(t)$, as is pictured in Figure \ref{fig:yy-reduction}.

\begin{lemma}
	\label{lem:relations_of_degrees_of_dots_for_g}
	Let $w$ be a Christoffel word and let $g_i = (a_i,b_i)$ be the dots contributing to $G_w(t)$. The degrees are $\deg(g_1) = 0$ and $\deg(g_{i+1}) = \deg(g_i) - (-1)^{b_i}$.
\end{lemma}
\begin{proof}
	The statement is vacuous if $G_w(t) = 0$. So assume the Christoffel word $w$ contains the factor $yy$ at least once. This implies (see Lemma 2.10 of \cite{berstelcombinatorics}) that $xx$ does not occur in $w$. Thus factors of $w$ between consecutive occurences of $yy$ \footnote{Note that e.g. the word $yyy$ has two occurences of $yy$.} must be of the form $yy(xy)^ny$ for $n \geq 0$. Similarly, $xy(xy)^my$ for $m \geq 0$ must be a prefix of $w$. The claims follow from an evaluation of these cases, pictured in Figure \ref{fig:proof_of_technical_lemma_for_G}.
\end{proof}

\begin{figure}
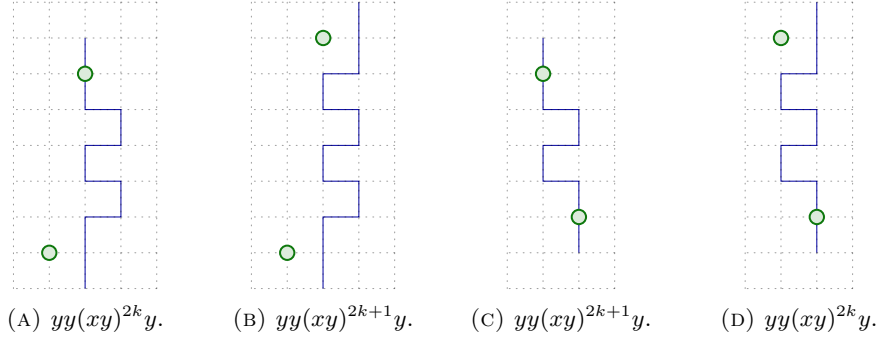

	\setbox1=\hbox{\includesvg{figures/Fig__07_LEMMA_for_G,_n_not_0,_even,_-x_step.svg}}
	\subcaptionbox{$yy(xy)^{2k}y$.\label{subfig:Lemma_degrees_of_greendots_subcase2}}[.25\textwidth]{
		\raisebox{\dimexpr\ht1-\height}{\includesvg{figures/Fig__05_Lemma_for_G,_n_not_0,_even,_+x_step.svg}}
	}%
	\subcaptionbox{$yy(xy)^{2k+1}y$.\label{subfig:Lemma_degrees_of_greendots_subcase3}}[.25\textwidth]{
		\raisebox{\dimexpr\ht1-\height}{\includesvg{figures/Fig__04_Lemma_for_G,_n_not_0,_odd,_+x_step.svg}}
	}%
	\subcaptionbox{$yy(xy)^{2k+1}y$.\label{subfig:Lemma_degrees_of_greendots_subcase6}}[.25\textwidth]{
		\raisebox{\dimexpr\ht1-\height}{\includesvg{figures/Fig__07_LEMMA_for_G,_n_not_0,_even,_-x_step.svg}}
	}%
	\subcaptionbox{$yy(xy)^{2k}y$.\label{subfig:Lemma_degrees_of_greendots_subcase5}}[.25\textwidth]{
		\raisebox{\dimexpr\ht1-\height}{\includesvg{figures/Fig__06_LEMMA_for_G,_n_not_0,_odd,_-x_step.svg}}
	}
	\caption{Proof of Lemma \ref{lem:relations_of_degrees_of_dots_for_g}. A dot $g = (a,b) \in SC_w$ lies on the folded path (is shifted left) if $b$ is even (odd, respectively).\\
	\ref{subfig:Lemma_degrees_of_greendots_subcase2},\ref{subfig:Lemma_degrees_of_greendots_subcase3}: $\deg(g_{i+1}) = \deg(g_i) + 1$, and
	\ref{subfig:Lemma_degrees_of_greendots_subcase6},\ref{subfig:Lemma_degrees_of_greendots_subcase5}: $\deg(g_{i+1}) = \deg(g_i) - 1$.
	\label{fig:proof_of_technical_lemma_for_G}}
\end{figure}

\subsection{Proof of the Main Theorem}
\label{subsec:proof_delta_log_concave}
We are now almost ready to tackle the proof of log-concavity of the polynomial $\Delta_w(t)$. Let $w = x\Pal(v)y$ be a Christoffel word. Our proof will be by induction on the length of $v$. In the inductive step of the argument,  Corollary \ref{cor:ratio_dominance_for_delta} is applied if $v = xv'$ and either Proposition \ref{prop:delta_of_yxv_reflection} or \ref{prop:g_as_delta_n_geq_2} if $v = yv'$. Informally, the conclusions and assumptions of Corollary \ref{cor:ratio_dominance_for_delta} versus Proposition \ref{prop:delta_of_yxv_reflection} or \ref{prop:g_as_delta_n_geq_2} play ``ping-pong'' in our inductive argument - the assumptions required for log-concavity of $\Delta_w(t)$ in Corollary \ref{cor:ratio_dominance_for_delta} follow from the conclusion of Proposition \ref{prop:delta_of_yxv_reflection} or \ref{prop:g_as_delta_n_geq_2}, and vice versa.

\begin{figure}[h]
	\centering
	\includesvg{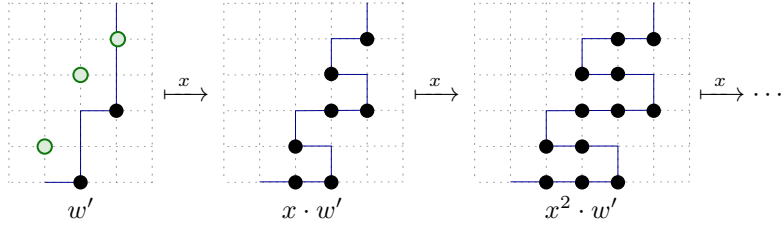}
	\caption{A horizontal segment of $w'$ has length $0$ or $1$. The corresponding horizontal segment of $w = x^n \cdot w'$ has length $n$ or $n+1$ and its leftmost point agrees with that of $w'$.}
	\label{fig:proof_of_alex_as_sum_convolution}
\end{figure}

\begin{proposition}
	\label{prop:alex_as_sum_of_convolution}
	Let $w' = x\Pal(yv')y$ for some $v' \in \{x,y\}^*$ and let $n \geq 1$. If $w = x^n \cdot w'$, where the action of $\{x,y\}^*$ on Christoffel words is as defined in Equation \ref{eq:monoid_action_on_christoffel}, then
	$$\Delta_w(t) = (1 + \dots + t^n)\Delta_{w'}(t) + (1 +  \dots + t^{n-1}) tG_{w'}(t).$$
\end{proposition}
\begin{proof}
	Let $w = x^n \cdot w' = x\Pal(x^nyv')y$. Recall that $x^n$ acts on $u' = \Pal(yv')$ by increasing the length of horizontal segments of $u'$ by $n$. The horizontal segments of $u' = \Pal(yv')$ have lengths either $0$ or $1$, as the the first letter of $yv'$ is $y$. Therefore the horizontal segments of $u = x^n \cdot u'$ have lengths either $n$ or $n+1$. Further, the position of the leftmost point on each horizontal segment of $w$ agrees with the position of the leftmost point on the corresponding segment of $w'$, see Figure \ref{fig:proof_of_alex_as_sum_convolution}.

	First consider the horizontal segments of $w$ of length $n$. These correspond to horizontal segments of $w'$ of length $0$, which are in bijective correspondence with occurences of the factors $yx^0y = yy$ in $w'$, and thus in bijective correspondence with the dots contributing to $G_{w'}(t)$. Let $d$ be the dot in the folded Christoffel path for $w'$ under this correspondence. The $n$ dots on the horizontal segment of $w$ have degrees $\deg(d)+1, \dots, \deg(d)+n$ and therefore contribute $$t^{\deg(d)} (t+ \dots + t^n)$$ to $\Delta_w(t)$.
	\footnote{Note that the purpose of the shift in the definition of $G_w(t)$ is apparent here, namely to account for whether a horizontal segment consists of $x$-steps or $-x$-steps.}

	Next consider the horizontal segments of $w$ of length $n+1$. These correspond to horizontal segments of $w'$ of length $1$, i.e. to the $\pm x$-steps of $w'$, and hence bijectively to the dots contributing to $\Delta_w'(t)$. Let $d$ be the corresponding dot in the folded Christoffel path of $w'$. The degrees of the $n+1$ dots on the horizontal segment of $w$ are $\deg(d), \dots, \deg(d)+n$ and thus their contributions to $\Delta_w(t)$ are $$t^{\deg(d)} (1 + \dots + t^n).$$
	The claim now follows from Equation \ref{eq:polynomial_as_dots}, as
	\begin{align*}
		\Delta_w(t) &= \sum_{d \in T_w} t^{\deg(d)} \\
					&= \sum_{d \in T_{w'}} t^{\deg(d)} (1 + \dots + t^n) + \sum_{d \in SC_{w'}} t^{\deg(d)} (t+ \dots + t^n) \\
					&= (1 + \dots + t^n)\Delta_{w'}(t) + (1 +  \dots + t^{n-1}) tG_{w'}(t). \qedhere
	\end{align*}
\end{proof}

\begin{cor}
	\label{cor:ratio_dominance_for_delta}
	Let $w' = x\Pal(yv')y$ for some $v' \in \{x,y\}^*$. If $\Delta_{w'} \lesssim t G_{w'}$, then $$\Delta_{x^n \cdot w'}(t) \lesssim \Delta_{x^{n+1} \cdot w'}(t).$$
\end{cor}
\begin{proof}
	We apply Fact \ref{fact:convolution_ratio_dominance_and_lexiocographic}. By the previous lemma,
	\begin{align*}
		\Delta_{x^n \cdot w'}(t) &= \underbrace{(1 + \dots + t^n)}_{A_1} \underbrace{\Delta_{w'}(t)}_{W_1} + \underbrace{(1 +  \dots + t^{n-1})}_{A_2} \underbrace{t G_{w'}(t)}_{W_2}, \mbox{~and~} \\
		\Delta_{x^{n+1} \cdot w'}(t) &= \underbrace{(1 + \dots + t^{n+1})}_{B_1} \underbrace{\Delta_{w'}(t)}_{W_1} + \underbrace{(1 +  \dots + t^{n})}_{B_2} \underbrace{t G_{w'}(t)}_{W_2}.
	\end{align*}
	By assumption, $W_1 \lesssim W_2$. The equations \ref{eq:lexiocographic_function_1} and \ref{eq:lexiocographic_function_2} are of the forms
	\begin{equation*}
		\frac{1}{1} \leq \frac{1}{1}, \mbox{~or~} \frac{1}{1} \leq \frac{0}{0}, \mbox{~or~}\frac{1}{0} \leq \frac{0}{0},
	\end{equation*}
	and therefore satisfied (see the remark immediately preceeding Fact \ref{fact:convolution_ratio_dominance_and_lexiocographic}).
\end{proof}

The following two propositions explain the relation between $\Delta_w(t)$ and $G_w(t)$ for Christoffel words $w$ of slope $1 < \frac{p}{q}$, i.e. precisely those Christoffel words satisfying $G_w(t) \neq 0$. Proposition \ref{prop:delta_of_yxv_reflection} treats the case of Christoffel words of slope $1 < \frac{p}{q} < 2$ and Proposition \ref{prop:g_as_delta_n_geq_2} those of slope $2 < \frac{p}{q}$.

\begin{proposition}
	\label{prop:delta_of_yxv_reflection}
	Let $w = x\Pal(yxv)y$ for some $v \in \{x,y\}^*$. Then
	\begin{equation*}
		\Delta_w(t) = t^2 \Delta_{x \cdot w'} (t^{-1}) \mbox{~and~} G_w(t) = t \Delta_{w'}(t^{-1}),
	\end{equation*}
	where $w' = x\Pal(E(v))y$. In particular, the coefficient sequences of $\Delta_w(t)$ and $G_w(t)$ are the reverses of the coefficient sequences of $\Delta_{w'}(t)$ and $\Delta_{x \cdot w'}(t)$.
\end{proposition}

\begin{proposition}
	\label{prop:g_as_delta_n_geq_2}
	Let $w = x\Pal(y^nxv)y$ for $v \in \{x,y\}^*$ and $n \geq 2$. Then
	\begin{equation*}
		\Delta_w(t) = \Delta_{w'}(t) \mbox{~and~} t G_w(t) = \Delta_{x \cdot w'}(t),
	\end{equation*}
	where $w' = x\Pal(y^{n-2}xv)y$.
\end{proposition}

\missingfigure{Create a Figure with a suitable picture for either Proposition \ref{prop:delta_of_yxv_reflection} or \ref{prop:g_as_delta_n_geq_2}.}

\begin{proof}[Proof of Proposition \ref{prop:delta_of_yxv_reflection}]
	We begin by showing that $\Delta_w(t) = t^2 \Delta_{x \cdot w'} (t^{-1})$. Since $x \cdot w' = x\Pal(xE(v))y= r(w)$, the $yy$-reduction of $w$, this claim follows from Lemma \ref{lem:delta_of_yy_reduction}.
	It remains to prove that $G_w(t) = t \Delta_{w'}(t^{-1})$. Let $n \geq 1$, let $r_i \geq 1$ for $1 \leq i \leq n$ and write
	\begin{align*}
		\Pal(yxv) &= y \cdot \left ( x^{r_1}yx^{r_2}y \cdots yx^{r_n} \right ), \\
		\Pal(E(v)) &= y^{r_1-1}xy^{r_2-1}x \cdots xy^{r_n-1}.
	\end{align*}
	Let $SC_w = g_1, \dots, g_k$ be the dots contributing to $G_w(t)$ and let $T_{w'} = d'_1, \dots, d'_l$ be the dots contributing to $\Delta_{w'}(t)$. We show that $l=k$ and $\deg(g_i) = 1 - \deg(d'_i)$; Equation \ref{eq:polynomial_as_dots} then implies that $G_w(t) = t \Delta_{w'}(t^{-1})$. The number of $yy$-steps in $w = y \cdot x\Pal(xv)y$ equals the number of $y$-steps in $x\Pal(xv)y$, which is $k = n$ and the number of $x$-steps in $w' = x\Pal(E(v))y$ is $l = n$. For $i = 1$, Lemma \ref{lem:relations_of_degrees_of_dots_for_g} shows that $\deg(g_1) = 0 = 1 - \deg(d'_1)$. Let $g_i = (*, b_i)$ and $d'_i = (*, b'_i)$. By Lemma \ref{lem:relations_of_degrees_of_dots_for_g} and immediately from the definitions,
	\begin{align*}
		\deg(g_{i+1}) &= \deg(g_i) - (-1)^{b_i} = \deg(g_i) - (-1)^{\sum_{j = 1}^i (r_j + 1)}, \\
		\deg(d'_{i+1}) &= \deg(d'_i) + (-1)^{b'_{i+1}} = \deg(d'_i) + (-1)^{\sum_{j = 1}^i (r_j - 1)}.
	\end{align*}
	The claim $\deg(g_i) = 1 - \deg(d'_i)$ now follows from induction on $i$. \qedhere
\end{proof}

\begin{figure}[h]
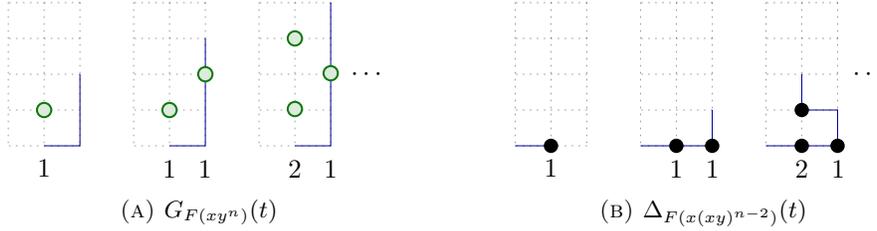

	\setbox2=\hbox{\includesvg{figures/Proof_of_main_theorem_10_G_of_xyy...y.svg}}%
	\subcaptionbox{$G_{F(xy^n)}(t)$ \label{subfig:G_of_xy^n}}[.47\textwidth]{
		\raisebox{\dimexpr\ht2-\height}{\includesvg{figures/Proof_of_main_theorem_10_G_of_xyy...y.svg}}
	}\hfill%
	\subcaptionbox{$\Delta_{F(x(xy)^{n-2})}(t)$ \label{subfig:Delta_of_x(xy)^{n-2}}}[.47\textwidth]{
		\raisebox{\dimexpr\ht2-\height}{\includesvg{figures/Proof_of_main_theorem_11_Delta_of_xxy...xy.svg}}
	}%
	\caption{$t G_{F(xy^n)}(t) = \Delta_{F(x(xy)^{n-2})(t)}$}
	\label{fig:g_as_delta_n_geq_2}
\end{figure}

\begin{proof}[Proof of Proposition \ref{prop:g_as_delta_n_geq_2}]
	Let $w,w'$ be as in the statement of the proposition. We begin by showing that $\Delta_w(t) = \Delta_{w'}(t)$. This follows from Corollary \ref{cor:delta_invariant_under_killing_yy}, as $w'$ is obtained from $w$ by deleting occurences of $yy$. The second claim of the Proposition is that $t G_w(t) = \Delta_{x \cdot w'}(t)$. Let $n \geq 2$ and $r_i \geq 1$ for $1 \leq i \leq m$ and write
	\begin{align*}
		\Pal(xv) &= x^{r_1}yx^{r_2}\cdots yx^{r_m},\\
		y^n \cdot \Pal(xv) &= (y^n x)^{r_1} y (y^n x)^{r_2} \cdots y (y^n x)^{r_m} y^n.
	\end{align*}
	This presents the Christoffel word $w = x\Pal(y^n xv)y$ as a product of factors of the form $xy^i$ for $i \geq 2$, namely as 
	\begin{align*}
		w &= x ( (y^n x)^{r_1} y (y^n x)^{r_2} \cdots y (y^n x)^{r_m} y^n ) y \\
			&= (xy^n) (xy^n)^{r_1 - 1} (xy^{n+1}) (xy^n)^{r_2 - 1}\cdots (xy^n)^{r_m - 1} (xy^{n+1}).
	\end{align*}
	Define a word $w' \in \{x,y\}^*$ by replacing each factor $xy^i$ of $w$ with $x(xy)^{i-2}$. We first employ Lemma \ref{lem:delta_and_g_of_a_concatenation} to show that $t G_w(t) = \Delta_{F(w')}(t)$ and then prove that $w' = x\Pal(y^{n-2}xv)y$. The following are easy to check, see Figure \ref{fig:g_as_delta_n_geq_2}.
	\begin{itemize}
		\item $t G_{F(xy^i)}(t) = \Delta_{F(x(xy)^{i-2})}(t)$,
		\item the terminal point of $F(xy^i)$ is $(1,i)$, and
		\item if $i$ is even [odd] then $F(x(xy^{i-2}))$ ends at $(1,i-2)$ [at $(2,i-2)$, respectively].
	\end{itemize}
	Now write $z = \prod_{i \in I} (xy^i)$ and $z' = \prod_{i \in I} (x(xy)^{i-2})$ and assume that $t G_{F(z)}(t) = \Delta_{F(z')}(t)$ for all products of at most $n = |I|$ factors. Using Lemma \ref{lem:delta_and_g_of_a_concatenation},
	\begin{align*}
		t G_{F(xy^i z)}
		&= t G_{F(xy^i)} + \begin{cases}
			t~tG_{F(z)}(t) & \text{for~} i \text{~even}\\
			t^2~t^{-1} G_{F(z)}(t^{-1}) & \text{for~} i \text{~odd}
		\end{cases} \\
		&= \Delta_{F(x(xy)^{i-2})}(t) + \begin{cases}
			t \Delta_{F(z')}(t) & \text{for~} i-2 \text{~even}\\
			t^2 \Delta_{F(z')}(t^{-1}) & \text{for~} i-2 \text{~odd}
		\end{cases} \\
		&= \Delta_{F(x(xy)^{i-2}z')}.
	\end{align*}
	The claim $t G_w(t) = \Delta_{F(w')}(t)$ now follows by induction. The final step is to show that $w' = x\Pal(y^{n-2}xv)y$. By definition, the word $w'$ is given by
	\begin{align*}
		w'	&= \left( x(xy)^{n-2} \right) {\left(x(xy)^{n-2} \right)}^{r_1 - 1}  \left(x(xy)^{n-1}\right) \cdots {\left(x(xy)^{n-2}\right)}^{r_m - 1} \left(x(xy)^{n-1} \right) \\
		&= x \left((xy)^{n-2}x \right) {\left((xy)^{n-2}x \right)}^{r_1 - 1} \left((xy)^{n-1}x \right) \cdots {\left((xy)^{n-2}x \right)}^{r_m - 1}  (xy)^{n-2}(xy) \\
		&= x \left [x \cdot \left ( y^{n-2}x \left ( y^{n-2}x \right )^{r_1 - 1} y^{n-1} \cdots \left (y^{n-2}x \right )^{r_m - 1} \right ) \right ] y \\
		&= x \left [ \left ( xy^{n-2} \right ) \cdot \left (x x^{r_1 -1}y \cdots yx x^{r_m - 1} \right ) \right ] y \\
		&= x \Pal(xy^{n-2}xv) y. \qedhere
	\end{align*}
\end{proof}

\begin{theorem}
	\label{thm:maintheorem}
	The polynomial $\Delta$ is log-concave for Christoffel words.
\end{theorem}
\begin{proof}
	The argument proceeds by induction on the length $|v'|$ in the expression $x\Pal(v')y$ for non-trivial Christoffel words. Specifically, we claim that $\Delta(t) \lesssim t G(t)$ for all Christoffel words.%
	\footnote{Recall that $f \lesssim g$ stipulates that both $f$ and $g$ are log-concave and that $f \lesssim 0$ is equivalent to $f$ being log-concave, as every log-concave sequence is ratio-dominated by the trivial sequence $(0,0,\cdots,0,\cdots)$.}%
	First consider $v' = x^n$ or $v' = y^n$ for $n \geq 0$. Both cases are easily dispatched with by explicit computation. We further remark that $$\Delta_{x\Pal(x^n)y}(t) = t + \dots + t^{n+1} \lesssim t + \dots + t^{n+2} = \Delta_{x\Pal(x^{n+1})y}(t).$$ Now assume $\Delta(t) \lesssim t G(t)$ for all Christoffel words $x\Pal(v')y$ whenever $|v'| \leq n$ and consider $w = x\Pal(v)y$ with $|v| = n+1$. If $v \neq x^{n+1}$ and $v \neq y^{n+1}$ then $v = x^k (yv')$ or $v = y^k (xv')$ and we thus need to consider the following cases.
	\begin{enumerate}
		\item $v = x^k yv'$. Let $w' = x\Pal(yv')y$ so that $w = x^k \cdot w'$. By induction, $\Delta_{w'}(t) \lesssim t G_{w'}(t)$ and Corollary \ref{cor:ratio_dominance_for_delta} implies that $\Delta_w(t)$ is log-concave. Observe that $G_w(t) = 0$ and therefore this is equivalent to $\Delta_w(t) \lesssim t G_w(t)$.
		\item $v = y xv'$. Let $w' = x\Pal(E(v'))y$. Note that $f(t) \lesssim g(t)$ implies $g(t^{-1}) \lesssim f(t^{-1})$ as the sequence of the ratios of the coefficients of $f(t^{-1})$ is the reverse sequence of the reciprocals of those of $f(t)$ and $x \mapsto x^{-1}$ is order-reversing. Either by the remark at the beginning of this proof (if $E(v') = x^m$) or else by induction and Corollary \ref{cor:ratio_dominance_for_delta} for the ratio-dominance, and by Proposition \ref{prop:delta_of_yxv_reflection} for the equalities,
		$$\Delta_w(t) = t^2 \Delta_{x \dot w'}(t^{-1}) \lesssim t^2 \Delta_{w'}(t^{-1}) = tG_w(t).$$
		\item $v = y^k xv'$ for $k \geq 2$. Let $w' = x\Pal(y^{k-2}xv')y$. By Proposition \ref{prop:g_as_delta_n_geq_2} for the equalities, and either by the remark above (if $n = 2$ and $v' = x^m$) or by induction and Corollary \ref{cor:ratio_dominance_for_delta} for the ratio-dominance,
		\begin{equation*}
			\Delta_w(t) = \Delta_{w'}(t) \lesssim \Delta_{x \cdot w'}(t) = t G_w(t). \qedhere
		\end{equation*}
	\end{enumerate}
\end{proof}

Theorem \ref{thm:maintheorem} does not hold for arbitrary folded paths, not even for folded paths arising from words of the form $v = xuy$ for $u$ a palindrome. For example, as is pictured in Figure \ref{fig:counterex_to_trapezoidal_for_palindromes}, the polynomial $\Delta_{F(v)}(t)$ may not even be unimodal in this case.

\begin{figure}[h]
	\centering
	\includesvg{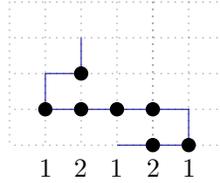}
	\caption{The polynomial $\Delta_{F(v)}(t) = t^{-2} + 2t^{-1} + 1 + 2t + t^2$ for $v = x(xyx^4yx)y$ is not unimodal.}
	\label{fig:counterex_to_trapezoidal_for_palindromes}
\end{figure}

\section{The Strong Fox Conjecture}
\label{sec:strongfox}

We return to our original question: Are Alexander polynomials of alternating knots log-concave? In this section, we give a positive answer for two-bridge knots.

Two-bridge links are alternating links parametrized by a pair of coprime integers $(p,q)$ satisfying $0 < p < q$ and $p$ odd. We write $D(p,q)$ for the corresponding two-bridge link. For example, Figure \ref{fig:two_bridge_knot_p=5,q=8} pictures the two-bridge link $D(5,8)$. If $q$ is odd, then $D(p,q)$ is a knot and if $q$ is even, then $D(p,q)$ is a 2-component link. Two-bridge knots are alternating and their branched double covers are Lens spaces.

The Alexander polynomial $\Delta_K(t) \in \mathbb{Z}[t,t^{-1}]$ of a knot $K \subset S^3$ is a polynomial defined up to multiplication by a unit in $\mathbb{Z}[t]$. It is convenient to consider the so-called balance class of the Alexander polynomial, where two polynomials $f(t),g(t)$ are said to be balanced if $\pm t^n f(t) = g(t)$ or $f(t) = \pm t^n f(t)$ for some non-negative integer $n$ \cite{MR907872}. We write $f \doteq g$ to indicate that $f,g$ are balanced.

Minkus gives the following expression for the Alexander polynomial of a two-bridge link. His argument proceeds by first finding a Wirtinger presentation of the link group of $D(p,q)$ and applying Fox calculus.
\begin{fact}[Lemma 11.1 of \cite{MR643587}]
	\label{fact:minkus-alexander}
	Let $p,q$ be coprime, non-negative integers satisfying $0 < p < q$ and $p$ odd. The Alexander polynomial of $D(p,q)$ is 
	\begin{equation}
		\label{eq:alex_using_epsilon_sequence}
		\Delta_{D(p,q)}(t) \doteq \sum_{k = 0}^{q-1} (-1)^k t^{\sum_{i=0}^k \epsilon_i},
	\end{equation}
	where $\epsilon_i = (-1)^{\lfloor \frac{ip}{q} \rfloor}$ and $\lfloor x \rfloor$ is the largest integer less than or equal to $x$.
\end{fact}

Let us briefly discuss the signs in Equation \ref{eq:alex_using_epsilon_sequence}. Recall the sign-alternation operator $\varsigma$ for Laurent polynomials from Section \ref{sec:introduction}, mapping $t^n \mapsto (-1)^n t^n$. The operator $\varsigma$ reverses the sign of every other coefficient of the Laurent polynomial. Now let $n = \sum_{i=0}^k \epsilon_i$. Clearly, $k \equiv n \Mod{2}$ since $\epsilon_i = \pm 1$.  Therefore, $(-1)^k t^n = (-1)^n t^n$ and it follows that 
\begin{equation}
	\label{eq:alex_using_epsilon_and_sign_alternation}
	\Delta_{D(p,q)}(t) \doteq \varsigma \left( \sum_{k = 0}^{q-1} t^{\sum_{i=0}^k \epsilon_i} \right) .
\end{equation}

The right side of Equation \ref{eq:alex_using_epsilon_and_sign_alternation} can be interpreted in terms of the polynomial $\Delta_w(t)$ for a Christoffel word $w$. 

\begin{lemma}
	\label{lem:polynomial_via_epsilon_sequences}
	Let $p,q$ be coprime, non-negative integers and let $w$ be the Christoffel word of slope $\frac{p}{q}$. The polynomial $\Delta_w(t)$ is given by
	\begin{equation*}
		\label{eq:christoffel_using_epsilon_sequence}
		\Delta_w(t) = \sum_{k = 0}^{q-1} t^{\sum_{i=0}^k \epsilon_i},
	\end{equation*}
	where $\epsilon_i = (-1)^{\lfloor \frac{ip}{q} \rfloor}$ and $\lfloor x \rfloor$ is the largest integer less than or equal to $x$.
\end{lemma}
\begin{proof}
	We use the formula for $\Delta_w(t)$ given in Equation \ref{eq:polynomial_as_dots}, $$\Delta_w(t)= \sum_{d \in T_w} t^{\deg(d)}.$$
	Let $d_1, \dots, d_q$ be the dots contributing to $\Delta_w(t)$. We claim that $\deg(d_k) = \sum_{i=0}^{k-1} \epsilon_i$. For $k=0$, $\deg(d_1) = 1 = (-1)^0 = \epsilon_0$. Note that $\deg(d_{k+1}) = \deg{d_k} \pm 1$, where the sign depends on whether $d_{k+1}$ corresponds to an $x$-step ($+1$) or an $-x$-step ($-1$). According to the construction of the folded path in Definition \ref{def:folded_path}, this is determined by the parity of the count of $y$-steps preceeding $d_{k+1}$, i.e. the count of occurences of $y$ preceeding the $(k+1)$-th occurence of $x$. The claim and thus the lemma now follow from Theorem 6.12 of \cite{berstelcombinatorics}, which states this count is given by 
	\begin{equation*}
		\left |\{jq~|~ j = 1, \dots, p-1 \mbox { satisfying } jq < (k+1)p \}\right | = \left \lfloor \frac{(k+1)p}{q} \right \rfloor. \qedhere
	\end{equation*}
\end{proof}
Comparing Lemma \ref{lem:polynomial_via_epsilon_sequences} to Fact \ref{fact:minkus-alexander}, we deduce the following theorem and the strong Fox conjecture for two-bridge knots as corollary thereof.

\begin{theorem}
	\label{thm:two_bridge_alex_is_delta_of_christoffel}
	Let $p,q$ be coprime integers satisfying $p$ odd and $0 < p < q$. If $L = D(p,q)$ and if $w$ is the Christoffel word of slope $\frac{p}{q}$, then the Alexander polynomial $\Delta_L(t)$ satisfies
	$$\Delta_L(t) \doteq \varsigma( \Delta_w(t)).$$ \qedhere 
\end{theorem}
\begin{cor}
	\label{cor:strong_fox_for_two_bridge}
	The strong Fox conjecture holds for two-bridge knots.
\end{cor}
\begin{proof}
	Let $K = D(p,q)$ be the two-bridge knot and $w$ the Christoffel word of slope $\frac{p}{q}$. Note that if $f(t) \doteq g(t)$ then $f(t)$ is log-concave if and only if $g(t)$ is log-concave, as the coefficient sequences of these polynomials are shifts of one another. Further, the sequence of absolute values of the coefficients of $\Delta_K(t) = \varsigma(\Delta_w(t))$ is the coefficient sequence of $\Delta_w(t)$. Theorem \ref{thm:maintheorem} asserts that $\Delta_w(t)$ is log-concave; thus by definition its coefficient sequence is log-concave.
\end{proof}

It is a classic fact that the two-bridge knots $D(p,q)$ and $D(p',q')$ are equivalent if and only if $q' = q$ and $p' = p^{\pm 1} \mod{q}$. This has the following interesting consequence for the polynomial $\Delta_w(t)$.

\begin{theorem}
	\label{thm:equivalent_delta_polynomials}
	Let $p < q$ be coprime odd integers, and suppose that $p'$ is also odd, where $p'$ is the standard representative of $p^{-1} \in \mathbb{Z}_q$. Let $w$ be the Christoffel word of slope $\frac{p}{q}$ and $w'$ the Christoffel word of slope $\frac{p'}{q}$. Then $\Delta_w(t) = \Delta_{w'}(t)$.
\end{theorem}
\begin{proof}
	Let $K = D(p,q) = D(p',q)$. The Alexander polynomial is a knot invariant, so by Theorem \ref{thm:two_bridge_alex_is_delta_of_christoffel}
	$$\Delta_w(t) \doteq \varsigma(\Delta_K(t)) \doteq \Delta_{w'}(t).$$
	Therefore $\Delta_w(t) = t^{n} \Delta_{w'}(t)$. It remains to show that $n=1$. Let $d_1, \dots, d_{q}$ and $d'_1, \dots, d'_q$ be the dots contributing to $\Delta_w(t)$ and $\Delta_{w'}(t)$, respectively. According to the proof of Lemma \ref{lem:polynomial_via_epsilon_sequences}, $\deg(d_k) = \sum_{i=0}^{k-1} \epsilon_i$. Since $p$ and $p'$ are odd, the numbers $\epsilon_i$ satisfy the symmetry $\epsilon_i = \epsilon_{q-1-i}$. 
	Thus there are constants $c,c'$ given by $$c = \deg(d_i) + \deg(d_{q+1-i}) \mbox{ and } c' = \deg(d'_i) + \deg(d'_{q+1-i}.)$$ We argue that $\deg(d_q) = \deg(d'_q)$. This implies $c = c'$ and therefore $n=1$. Define a map $$f : \{ip~|~i=0,\dots,q-1\} \mapsto \{ip'~|~i=0,\dots,q-1\}$$ as follows: Write $ipp' = l(p'q) + r$ for $0 \leq r < p'q$ and map $ip \mapsto r$. It is standard to verify that $f$ is well-defined and bijective. Note that $\left\lfloor \frac{ipp'}{p'q} \right\rfloor = l$. We argue that
	\begin{equation}
		\label{eq:proof_of_equivalent_delta}
		\epsilon_i = (-1)^{\left\lfloor \frac{ip}{q} \right\rfloor} = (-1)^{\left\lfloor \frac{ipp'}{p'q} \right\rfloor} = (-1)^{\left\lfloor \frac{r}{q} \right\rfloor} = \epsilon'_{f(i)},
	\end{equation}
	which immediately implies that
	$$\deg(d_q) = \sum_{i=0}^{q-1} \epsilon_i = \sum_{i=0}^{q-1} \epsilon'_{f(i)} = \deg(d'_q).$$
	The only non-trivial equality in Equation \ref{eq:proof_of_equivalent_delta} is the second to last. By assumption, $pp' = 1 + kq$. As $p,p',q$ are odd, $k$ is even. Now, $l(p'q) + r = ipp' = i(kq + 1) = ikq + i$, and therefore $r = ikq - lp'q + i = q(ik - lp') + i$ for $0 \leq i < q$. In particular, $\lfloor \frac{r}{q} \rfloor = ik - lp'$. Observe that $l \equiv ik - lp' \mod{2}$, as $k$ is even and $p'$ is odd. This concludes the proof as
	\begin{equation*}
		(-1)^{\left\lfloor \frac{ipp'}{p'q} \right\rfloor} = (-1)^l = (-1)^{ ik - lp'} = (-1)^{\left\lfloor \frac{r}{q} \right\rfloor}. \qedhere
	\end{equation*}
\end{proof}
\section*{Acknowledgments}
The author would like to thank Jim Hoste, Pierre-Vincent Koseleff and Christophe Reutenauer for insightful comments and suggestions.

\bibliographystyle{halpha}
\bibliography{Alexander_Polynomial_Christoffel_words}

\end{document}